\documentclass[12pt,oneside,reqno]{scrartcl}

\usepackage{cihan}
\usepackage{cases}
\usepackage{yfonts}
\usepackage[affil-it]{authblk}
\usepackage{bbm}

\author{Cihan Bahran}
\affil{Department of Mathematics, Bilkent University\\
Ankara 06800, Turkey\\
cihan.bahran@bilkent.edu.tr}
\date{}
\usepackage{etoolbox}
 \usepackage{relsize}
 
\usepackage[scr=boondoxo,scrscaled=1.1]{mathalfa} 
 
\usepackage{bookmark}
\usepackage{textcomp}
\usepackage{centernot}

\usepackage{appendix}
\usepackage{xstring} 

\AtBeginDocument{%
  \addtolength\abovedisplayskip{-0.5\baselineskip}%
  \addtolength\belowdisplayskip{-0.5\baselineskip}%
}

\title{Regularity of an \texorpdfstring{$\FI$}{Lg}-module via the derivative functor}

\newcommand{\FI}{\mathbf{FI}}

\DeclareMathOperator{\crit}{crit}

\newcounter{dummy}
\makeatletter
\newcommand\sitem[1][]{\item[(#1)]\refstepcounter{dummy}\def\@currentlabel{#1}}
\makeatother

\newcommand{\cofi}[1]{\co_{#1}^{\FI}}

\newcommand{\width}[2]{\operatorname{width}_{\,#1}^{\deriv}(#2)}

\DeclareMathOperator{\reg}{reg}

\newcommand{\tgen}{t_{0}}
\newcommand{\trel}{t_{1}}
\newcommand{\shift}[2]{\mathbf{\Sigma}^{#2}   #1 }

\newcommand{\deriv}{\mathbf{\Delta}}
\newcommand{\kert}{\mathbf{K}}

\newcommand{\locoh}[1]{\co_{\mathfrak{m}}^{#1}}

\AtBeginDocument{%
   \def\MR#1{}
}

\makeatletter
\def\blfootnote{\gdef\@thefnmark{}\@footnotetext}
\makeatother

\begin{document}
\ytableausetup{centertableaux}
\maketitle

\blfootnote{\textup{2010} \textit{Mathematics Subject Classification}.
Primary 18A25.} 
\vspace{-0.8 in}
\begin{onecolabstract} 
We characterize the regularity of an $\FI$-module using the derivative functors.
\end{onecolabstract}
\vspace{0.4 in}

{\tableofcontents}

\section{Introduction}

An $\FI$-module is a functor $\FI \rarr \lMod{\zz}$ where $\FI$ is the category of finite sets and injections and $\lMod{\zz}$ is the category of abelian groups. We write $\lMod{\FI}$ for the functor category $[\FI, \lMod{\zz}]$. Given an $\FI$-module $W$, we write 
\begin{align*}
 \deg(W) &:= \min\{d \geq -1 : W_{S} = 0 \text{ for } |S| > d\} 
 \\
 &\in \{-1,0,1,2,3,\dots\} \cup \{\infty\} \, .
\end{align*}

\paragraph{$\FI$-homology and regularity.} A systematic way of investigating the ``unstable'' part of an $\FI$-module $V$ is to study its quotient $\cofi{0}(V)$ defined by
\begin{align*}
 \cofi{0}(V)_{S} := \coker\! \left(
 \bigoplus_{A \subsetneq S} V_{A} \rarr V_{S}
 \right) \, .
\end{align*}
The functor $\cofi{0} \colon \lMod{\FI} \rarr \lMod{\FI}$ is right exact and has left derived functors $\cofi{i} := \operatorname{L}_{i}\!\cofi{0}$.
Setting $t_{i}(V) : = \deg\!\left(\cofi{i}(V)\right)$, we say $V$ is \textbf{presented in finite degrees} if $\tgen(V)$ and $\trel(V)$ are finite. We also set the \textbf{regularity} of $V$ as
\begin{align*}
 \reg(V) &:= \max\!\left\{t_{i}(V) - i : i \geq 1\right\} 
 \\
 &\in \{-2,-1,0,1,\dots\} \cup\{\infty\}
\end{align*}
which is finite whenever $V$ is presented in finite degrees \cite[Theorem A]{ce-homology}.

\paragraph{The derivative functor, iterated and derived.} Given any $\FI$-module $V$, we write $\shift{V}{}$ for the composition 
\begin{align*}
 \FI \xrightarrow{- \sqcup \{*\}} \FI \xrightarrow{V} \lMod{\zz} \, .
\end{align*}
and call it the \textbf{shift} of $V$. This procedure defines the \textbf{shift functor} 
\[
\shift{}{} \colon \lMod{\FI} \to \lMod{\FI} \, .
\]
The inclusions $S \emb S \sqcup \{*\}$ for all finite sets $S$ define  a natural transformation 
\[
 \iota \colon \id_{\lMod{\FI}} \to \shift{}{} \, ,
\]
whose cokernel $\deriv := \coker(\iota)$
we call the \textbf{derivative functor}: it is a right exact functor
\[
 \deriv \colon \lMod{\FI} \rarr \lMod{\FI} \, .
\]
For every $i,a \geq 0$, we write $\co_{i}^{\deriv^{\!a}} := \operatorname{L}_{i}\!\deriv^{\!a}$ for the $i$-th left derived functor of the $a$-fold composite $\deriv^{\!a}$. 


The objective of this paper is to characterize the regularity of an $\FI$-module in terms of the derivative functor.

\begin{defn}
  Given an $\FI$-module $V$ and $i \geq 0$, we write
\begin{align*}
\width{i}{V} &:= 
\sup\!\left\{
  \deg\!\left(
           \co_{i}^{\deriv^{\!a}}(V)
        \right) 
  + a \,:\, a \geq 0 \text{ with } \co_{i}^{\deriv^{\!a}}(V) \neq 0
  \right\}
  \\
  & \in \{0,1,2, \dots\} \cup \{-\infty, \infty\}
\end{align*}
with the convention $\sup(\empt) = -\infty$.
\end{defn}

\begin{thmx}\label{width}
 Let $V$ be an $\FI$-module presented in finite degrees which is not $\cofi{0}$-acyclic. Then for every $i \geq 1$, we have 
\begin{align*}
 \width{i}{V} - i = \reg(V) \, .
\end{align*}
\end{thmx}

\paragraph{Relationship with previous literature.} The invariant $\width{1}{V}$ was first considered explicitly in \cite[Definition 3.7]{ramos-fig} (denoted $\partial\!\operatorname{width}$ there). 

The invariants $\width{i}{V}$ can be used to conceptually compartmentalize the proof of \cite[Theorem A]{ce-homology}, \cite[Theorem 3.19]{ramos-fig} bounding $\reg(V)$ in terms of $\tgen(V), \trel(V)$ into three parts as follows:

\begin{thm}[{\cite[Corollary 3.18]{ramos-fig}, based on \cite[Theorem E]{ce-homology}}] \label{width-bound}
 Every $\FI$-module $V$ presented in finite degrees satisfies
 \[
  \width{1}{V} \leq \trel(V) + \min\{\tgen(V), \trel(V)\} \, .
 \] 
\end{thm}

\begin{thm}[{\cite[proof of Theorem 4.8]{ce-homology}}]
 For every $\FI$-module $V$, there is a weakly decreasing sequence 
\[ 
\width{1}{V} - 1 \geq \width{2}{V} - 2 \geq \width{3}{V} - 3 \geq \cdots 
\]
\end{thm}

\begin{thm}[{\cite[proof of Theorem A, pages 2403-2404]{ce-homology}}]
 For every $\FI$-module $V$ with $\tgen(V) < \trel(V) < \infty$, we have
 \[
  \reg(V) \leq \max\left(
  \{\trel(V) - 1\} \cup
  \{\width{i}{V} - i : i \geq 2\}
  \right) \, .
 \]
\end{thm}

These are what partially led me to think something like Theorem \ref{width} could be true, collapsing the inequalities in the above two results to equalities. It is also perhaps worth mentioning that combining \cite[Theorem 4.7, part (2)]{ramos-fig}, \cite[Theorem A]{bahran-reg}, \cite[Corollary 2.9]{bahran-reg} with Theorem \ref{width}, the bound in Theorem \ref{width-bound} can be sharpened into the following:

\begin{thm}
 Given integers $a,b \geq -1$ and $i \geq 1$, we have
 \[ 
  \max\!\left\{\!
\width{i}{V} \!: \!\!
\begin{array}{l}
 \text{$V$ is an $\FI$-module with}
 \\
 \tgen(V) \leq a \, , \trel(V) \leq b 
\end{array}
 \!\!\right\} \,=\, 
\begin{cases}
 -\infty & \text{if $a=-1$ or $b \leq 0$,} \\
 i + a + b  - 1 \, & \text{if\, $0 \leq a < b$,} \\
 i + 2b - 2 & \text{if\, $a \geq b \geq 1$.} 
\end{cases}  
 \]
\end{thm}

Another result which can be interpreted in terms of $\width{1}{V}$ is about polynomial conditions on $\FI$-modules:

\begin{thm}[{\cite[proof of Proposition 4.18]{rw-wahl-stab}}] \label{rww-width}
 Let $t,w \geq -1$ be integers and $V$ an $\FI$-module with $\tgen(V) \leq t$ and $\width{1}{V} \leq w$. Then in the sense of \emph{\cite[Definition 4.10]{rw-wahl-stab}}, $V$ is of degree $t$ at $w$. 
\end{thm}

Due to Theorem \ref{rww-width}, I had initially thought that establishing Theorem \ref{width} for $i=1$ would be a necessary step to prove the characterization of polynomiality in \cite[Theorem B]{bahran-polynomial}, but later found a more direct argument during writing that paper. The current paper is a result of revisiting the invariants $\width{i}{V}$.

\section{Homological algebra of $\FI$-modules}
\subsection{Preliminaries}
In this section we recall relevant homological results from the $\FI$-module literature. The first one is a short exact sequence that facilitates inductive arguments about the derived functors $\co_{i}^{\deriv^{\!a}}$.

\paragraph{The kernel functor.} Recall from the introduction that the inclusions $S \emb S \sqcup \{*\}$ for all finite sets $S$ define  a natural transformation 
\[
 \iota \colon \id_{\lMod{\FI}} \to \shift{}{} \, ,
\]
whose cokernel is the derivative functor $\deriv$. For its kernel, we write $\kert := \ker(\iota)$: it is a left exact functor
\[
 \kert \colon \lMod{\FI} \rarr \lMod{\FI} \, .
\]

\begin{prop}[{\cite[(18)]{ce-homology}}] \label{ce-seq} For every $\FI$-module $V$ and integers $i,a \geq 1$, there is a short exact sequence
\begin{align*}
 0 \rarr \deriv\!\left(
 \!\co_{i}^{\deriv^{\!a-1}}\!(V)
 \right) \rarr \co_{i}^{\deriv^{\!a}}\!(V) 
 \rarr \kert\!\left(
 {\co_{i-1}^{\deriv^{\!a-1}}}(V) 
 \right)\rarr 0
\end{align*}
of $\FI$-modules.
\end{prop}

\paragraph{Torsion and local cohomology.} We say an $\FI$-module $V$ is \textbf{torsion} if for every finite set $S$ and $x \in V_{S}$, there is an injection $\alpha \colon S \emb T$ such that $V_{\alpha}(x) = 0 \in V_{T}$. We write
\begin{align*}
 \locoh{0} \colon \lMod{\FI} \rarr \lMod{\FI}
\end{align*}
for the functor which assigns an $\FI$-module its largest torsion $\FI$-submodule; it is left exact. For each $j \geq 0$, the $j$-th \textbf{local cohomology} functor is the $j$-th derived functor $\locoh{j} := \operatorname{R}^{j}\!\locoh{0}$ of $\locoh{0}$, for which we set 
\begin{align*}
 h^{j}(V) &:= \deg(\locoh{j}(V)) 
 \\
 &\in \{-1,0,1,\dots\} \cup \{\infty\} \, 
\end{align*}
for every $\FI$-module $V$.


We recall the characterization of regularity via local cohomology, which then justifies the well-definition of the critical index.

\begin{thm}[{\cite[Theorem 1.1, Remark 1.3]{nss-regularity}}]\label{nss}
 Let $V$ be an $\FI$-module presented in finite degrees which is not $\cofi{0}$-acyclic. Then the maximum of the finite set
\begin{align*}
 \empt \neq \{h^{j}(V) + j : h^{j}(V) \geq 0\} \subseteq \zz_{\geq0}
\end{align*} 
is equal to $\reg(V)$.
\end{thm}

\begin{defn} \label{defn:crit}
 For an $\FI$-module presented in finite degrees which is not $\cofi{0}$-acyclic, we define its \textbf{critical index} as
\begin{align*}
 \crit(V) := 
 \min\!\left\{
 j : h^{j}(V) \geq 0 \text{ and }
 h^{j}(V) + j = \reg(V)
 \right\} \, .
\end{align*}
\end{defn}

\subsection{Proving Theorem \ref{width}} \label{proof}
In this section we prove Theorem \ref{width}. As ingredients towards that end, we relate the vanishing degrees of the derived functors $\co_{i}^{\deriv^{\!a}}$ with the local cohomology degrees in Proposition \ref{careful} (the most technical part of the paper) and identify non-vanishing degrees for these functors in Proposition \ref{all-iso} using the critical index.

\begin{prop} \label{careful}
 For an $\FI$-module $V$ presented in finite degrees, the following hold: 
\begin{birki}
 \item For every $a \geq 0$ and $u \geq 0$, if $h^{u}(\deriv^{\!a}V) \geq 0$ then we have 
\begin{align*}
 h^{u}(\deriv^{\!a}V) + u + a
 \leq \max\!\left\{h^{j}(V) + j : h^{j}(V) \geq 0 \text{ and }
  0 \leq j \leq u + a\right\} \, .
\end{align*}
 \item For every $a \geq 1$ and $i \geq 1$, if $\co_{i}^{\deriv^{\!a}}(V) \neq 0$ then we have 
\begin{align*}
  \deg \co_{i}^{\deriv^{\!a}}(V) + a 
  \leq \max\!\left\{h^{j}(V) + j + i : h^{j}(V) \geq 0 \text{ and }
  0 \leq j \leq a-i\right\} \, .
\end{align*}
\end{birki}
\end{prop}

\begin{proof}
 (1) We employ induction on $a$: the base case $a=0$ is immediate so assume $a \geq 1$. By \cite[Proposition 2.6]{bahran-polynomial} applied to $\deriv^{\!a-1}V$, there is an exact sequence
\begin{align*}
 \shift{\!\locoh{u}\!\left(\deriv^{\!a-1}V\right)}{}
 \rarr
 \locoh{u}\!\left(\deriv^{\!a}V\right)
 \rarr
 \locoh{u+1}\!\left(\deriv^{\!a-1}V\right)
\end{align*}
of $\FI$-modules each of finite degree. If
$ \deg \locoh{u}\!\left(\deriv^{\!a-1}V\right) = h^{u}(\deriv^{\!a-1}V) \geq 0
$, then
\begin{align*}
 \deg \shift{\!\locoh{u}\!\left(\deriv^{\!a-1}V\right)}{} + u + a
 &= 
 \left(\deg \locoh{u}\!\left(\deriv^{\!a-1}V\right) -1 \right) + u + a
 \\
 &= h^{u}(\deriv^{\!a-1}V) + u + a - 1
 \\
 &\leq \max\!\left\{h^{j}(V) + j : h^{j}(V) \geq 0 \text{ and }
  0 \leq j \leq u + a-1\right\}
\end{align*}
by the induction hypothesis.  If
$ \deg \locoh{u+1}\!\left(\deriv^{\!a-1}V\right) = h^{u+1}(\deriv^{\!a-1}V) \geq 0
$, then
\begin{align*}
 \deg \locoh{u+1}\!\left(\deriv^{\!a-1}V\right) + u + a
 &= 
 h^{u+1}(\deriv^{\!a-1}V) + u + a
 \\
 &= h^{u+1}(\deriv^{\!a-1}V) + (u+1) + (a - 1)
 \\
 &\leq \max\!\left\{h^{j}(V) + j : h^{j}(V) \geq 0 \text{ and }
  0 \leq j \leq u + a\right\}
\end{align*}
again by the induction hypothesis. Therefore if $\deg \locoh{u}\!\left(\deriv^{\!a}V\right) = h^{u}(\deriv^{\!a}V) \geq 0$, by the exact sequence we have started with, we have that either
\begin{align*}
 \deg \locoh{u}\!\left(\deriv^{\!a-1}V\right) \geq 0
 \quad \text{or} \quad
 \deg \locoh{u+1}\!\left(\deriv^{\!a-1}V\right) \geq 0 \, .
\end{align*}
Thus 
\begin{align*}
 h^{u}(\deriv^{\!a}V) + u + a &\leq \max\!\left\{
 \deg \shift{\!\locoh{u}\!\left(\deriv^{\!a-1}V\right)}{} + u + a,\,
 \deg \locoh{u+1}\!\left(\deriv^{\!a-1}V\right) + u + a
 \right\}
 \\
 &\leq \max\!\left\{h^{j}(V) + j : h^{j}(V) \geq 0 \text{ and }
  0 \leq j \leq u + a\right\}
\end{align*}
as desired.

(2) Let us call the desired statement $P(a,i)$ and employ induction on $a+i$. For $a+i = 2$, we have $a=i=1$ so that the higher left derived functors of $\deriv^{a-1} = \id_{\lMod{\FI}}$ vanish, and Proposition \ref{ce-seq} reduces to an isomorphism $
 \co_{1}^{\deriv}\!(V) 
 \cong \kert(V) 
$. Indeed here
\begin{align*}
  \deg \co_{1}^{\deriv}(V) + 1 = \deg(\kert(V)) + 1 = h^{0}(V) + 1 \, ,
\end{align*}
verifying $P(1,1)$.
For $a+i \geq 3$, there are three cases:
\begin{itemize}
 \item $a=1$ and $i \geq 2$: the higher left derived functors of $\deriv^{a-1} = \id_{\lMod{\FI}}$ vanish, and Proposition \ref{ce-seq} reduces to the fact $
 \co_{i}^{\deriv}\!(V) = 0
$, therefore $P(1,i)$ is vacuously true.
\item $a \geq 2$ and $i=1$: By Proposition \ref{ce-seq} we have a short exact sequence
\begin{align*}
 0 \rarr \deriv\!\left(
 \co_{1}^{\deriv^{\!a-1}}\!(V)
 \right) \rarr \co_{1}^{\deriv^{\!a}}\!(V) 
 \rarr \kert\!\left(
 {\deriv^{\!a-1}}(V) 
 \right)\rarr 0 \, .
\end{align*}
If $\co_{1}^{\deriv^{\!a-1}}\!(V) \neq 0$, then 
\begin{align*}
 \deg \deriv\!\left(
 \co_{1}^{\deriv^{\!a-1}}\!(V)
 \right) + a
 &\leq \deg 
 \co_{1}^{\deriv^{\!a-1}}\!(V)
  + a-1
 \\
 &\leq \max\!\left\{h^{j}(V) + j + 1 : h^{j}(V) \geq 0 \text{ and }
 0 \leq j \leq a-2\right\}
\end{align*}
by the induction hypothesis $P(a-1,1)$. If $h^{0}\!\left(
 {\deriv^{\!a-1}}(V) 
 \right) \geq 0$, then
\begin{align*}
 \deg \kert\!\left(
 {\deriv^{\!a-1}}(V) 
 \right) + a 
 &= h^{0}\!\left(
 {\deriv^{\!a-1}}(V) 
 \right) + (a-1) + 1
 \\
 &\leq \max\!\left\{h^{j}(V) + j : h^{j}(V) \geq 0 \text{ and }
 0 \leq j \leq a-1\right\} + 1
\end{align*}
by part (1). Therefore if $\co_{1}^{\deriv^{\!a}}\!(V) \neq 0$, we have
\begin{align*}
 \deg \co_{1}^{\deriv^{\!a}}\!(V)
  + a
 &\leq \max\!\left\{h^{j}(V) + j + 1 : h^{j}(V) \geq 0 \text{ and }
 0 \leq j \leq a-1\right\} \, ,
\end{align*}
establishing $P(a,1)$.
\item $a \geq 2$ and $i \geq 2$: if $\co_{i}^{\deriv^{\!a-1}}\!(V) \neq 0$, then 
\begin{align*}
 \deg \deriv\!\left(
 \!\co_{i}^{\deriv^{\!a-1}}\!(V)
 \right) + a
 &\leq \deg \co_{i}^{\deriv^{\!a-1}}\!(V) + a - 1
 \\
 & \leq \max\!\left\{h^{j}(V) + j + i : h^{j}(V) \geq 0 \text{ and }
  0 \leq j \leq (a-1)-i\right\}
\end{align*}
by the induction hypothesis $P(a-1,i)$. If ${\co_{i-1}^{\deriv^{\!a-1}}}(V) \neq 0$, then
\begin{align*}
 &\,\,\,\,\,\,\,\deg \kert\!\left(
 {\co_{i-1}^{\deriv^{\!a-1}}}(V) 
 \right) + a 
 \\
 &\leq 
 \deg {\co_{i-1}^{\deriv^{\!a-1}}}(V) + a
 = \deg {\co_{i-1}^{\deriv^{\!a-1}}}(V) + (a-1) + 1
 \\
 &\leq \max\!\left\{h^{j}(V) + j + i-1 : h^{j}(V) \geq 0 \text{ and }
 0 \leq j \leq a-i\right\} + 1
\end{align*}
by the induction hypothesis $P(a-1,i-1)$. Therefore if $\co_{i}^{\deriv^{\!a}}(V) \neq 0$, by Proposition \ref{ce-seq} either $\co_{i}^{\deriv^{\!a-1}}\!(V) \neq 0$ or ${\co_{i-1}^{\deriv^{\!a-1}}}(V) \neq 0$, and 
\begin{align*}
  \deg \co_{i}^{\deriv^{\!a}}(V) + a &= \max\!\left\{
  \deg \deriv\!\left(
 \!\co_{i}^{\deriv^{\!a-1}}\!(V)
 \right) + a,\,
 \deg \kert\!\left(
 {\co_{i-1}^{\deriv^{\!a-1}}}(V) 
 \right) + a 
  \right\}
  \\
  &\leq \max\!\left\{h^{j}(V) + j + i : h^{j}(V) \geq 0 \text{ and }
  0 \leq j \leq a-i\right\},\,
\end{align*}
establishing $P(a,i)$.
\end{itemize}
\end{proof}

\begin{prop} \label{all-iso}
 Let $V$ be an $\FI$-module presented in finite degrees which is not $\cofi{0}$-acyclic. Setting $\rho := \reg(V)$ and $\gamma := \crit(V)$, the $\sym{\rho-\gamma}$-modules
\begin{birki}
 \item $\locoh{\gamma - j}(\deriv^{\!j}V)_{\rho-\gamma}$ for each $0 \leq j \leq \gamma$,
 \vspace{0.1cm}
 \item $\kert\!\left(\deriv^{\!\gamma}V\right)_{\rho-\gamma}$ ,
 \vspace{0.1cm}
 \item $\co_{i}^{\deriv^{\!\gamma+i}}(V)_{\rho-\gamma}$ for each $i \geq 1$,
\end{birki}
are all nonzero and pairwise isomorphic.
\end{prop}
\begin{proof}
 By Definition \ref{defn:crit}, we have $h^{\gamma}(V) \geq 0$ and $h^{\gamma}(V) + \gamma = \rho$, so that 
\begin{align*}
 0 \leq \deg \locoh{\gamma}(V) = h^{\gamma}(V) = \rho-\gamma
\end{align*}
and hence $\locoh{\gamma}(V)_{\rho - \gamma} \neq 0$. Next we will show that 
\begin{align*}
 \locoh{\gamma}(V)_{\rho - \gamma} 
 \cong 
 \locoh{\gamma - j}(\deriv^{\!j}V)_{\rho-\gamma}
\end{align*}
for every $0 \leq j \leq \gamma$ by induction on $\gamma = \crit(V)$. The base case $\gamma = 0$ is a tautology, so we assume $\gamma \geq 1$. By \cite[Proposition 2.10]{bahran-polynomial} the $\FI$-module $\deriv V$ is not $\cofi{0}$-acyclic with
\begin{align*}
 \reg(\deriv V) = \rho-1
 \quad \text{and} \quad
 \crit(\deriv V) = \gamma - 1 \, .
\end{align*}
Therefore we may apply the induction hypothesis to $\deriv V$ and get
\begin{align*}
 \locoh{\gamma - 1}(\deriv V)_{\rho - \gamma}
 &=
 \locoh{\gamma - 1}(\deriv V)_{(\rho-1) - (\gamma-1)} 
 \\ &\cong 
 \locoh{\gamma - 1 - j'}(\deriv^{\!j'}\deriv V)_{(\rho-1)-(\gamma-1)}
 \\ &=
 \locoh{\gamma - (j'+1)}(\deriv^{\!j'+1}V)_{\rho - \gamma}
\end{align*}
as $\sym{\rho-\gamma}$-modules for every $0 \leq j' \leq \gamma - 1$. Reindexing as $j = j'+1$,  
\begin{align*}
 \locoh{\gamma - 1}(\deriv V)_{\rho - \gamma} \cong \locoh{\gamma - j}(\deriv^{\!j}V)_{\rho - \gamma}
\end{align*}
for every $1 \leq j \leq \gamma$. By \cite[Proposition 2.6]{bahran-polynomial}, there is an exact sequence
\begin{align*}
 \shift{\!\locoh{\gamma-1}(V)}{} \rarr \locoh{\gamma-1}(\deriv V)
 \rarr \locoh{\gamma}(V) \rarr  \shift{\!\locoh{\gamma}(V)}{} 
\end{align*}
for which we have 
\begin{align*}
 \left(\shift{\!\locoh{\gamma-1}(V)}{}\right)_{\rho-\gamma} = \locoh{\gamma-1}(V)_{\rho-\gamma+1} = 0
\end{align*}
because $h^{\gamma-1}(V) + \gamma - 1 < \rho$, and
\begin{align*}
 \left( \shift{\!\locoh{\gamma}(V)}{} \right)_{\rho-\gamma}
 =
 \locoh{\gamma}(V)_{\rho-\gamma+1} = 0
\end{align*}
because $h^{\gamma}(V) + \gamma = \rho$ by the definition of $\gamma = \crit(V)$. Thus evaluating the exact sequence at degree $\rho-\gamma$ results in an isomorphism 
\begin{align*}
 \locoh{\gamma-1}(\deriv V)_{\rho-\gamma} 
 \cong \locoh{\gamma}(V)_{\rho-\gamma}
\end{align*}
of $\sym{\rho-\gamma}$-modules, completing the $j=0$ case of (1). 

For the isomorphism between the modules between (1) and (2), apply \cite[Proposition 2.6]{bahran-polynomial} to the $\FI$-module $\deriv^{\!\gamma} V$ and get an exact sequence 
\begin{align*}
 0 \rarr \kert\!\left(\deriv^{\!\gamma}V\right) 
 \rarr
 \locoh{0}\!\left(\deriv^{\!\gamma}V\right) 
 \rarr
 \shift{\!\locoh{0}\!\left(\deriv^{\!\gamma}V\right)}{} \, ,
\end{align*}
for which we have 
\begin{align*}
 \left(\shift{\!\locoh{0}\left(\deriv^{\!\gamma}V\right)}{}\right)_{\rho-\gamma} = \locoh{0}\!\left(\deriv^{\!\gamma}V\right)_{\rho-\gamma+1} = 0
\end{align*}
because by \cite[Proposition 2.10]{bahran-polynomial} we have $\crit(\deriv^{\!\gamma} V) = 0$ and
\begin{align*}
 h^{0}(\deriv^{\!\gamma} V) = \reg(\deriv^{\!\gamma} V) 
 =
 \reg(V) - \gamma = \rho - \gamma \, ,
\end{align*}
resulting in an isomorphism
\begin{align*}
 \kert\!\left(\deriv^{\!\gamma}V\right)_{\rho - \gamma} \cong
 \locoh{0}\!\left(\deriv^{\!\gamma}V\right)_{\rho - \gamma}
\end{align*}
of $\sym{\rho-\gamma}$-modules as desired.

Consider the short exact sequence
\begin{align*}
 0 \rarr \deriv\!\left(
 \!\co_{i}^{\deriv^{\!\gamma+i-1}}\!(V)
 \right) \rarr \co_{i}^{\deriv^{\!\gamma+i}}\!(V) 
 \rarr \kert\!\left(
 {\co_{i-1}^{\deriv^{\!\gamma+i-1}}}(V) 
 \right)\rarr 0 \tag{$\dagger$}
\end{align*}
of $\FI$-modules given by  Proposition \ref{ce-seq}. For $i=1$ it becomes
\begin{align*}
 0 \rarr \deriv\!\left(
 \co_{1}^{\deriv^{\!\gamma}}\!(V)
 \right) \rarr \co_{1}^{\deriv^{\!\gamma+1}}\!(V) 
 \rarr \kert\!\left(
 \deriv^{\!\gamma}(V) 
 \right)\rarr 0 \tag{$\dagger\dagger$} \, .
\end{align*}
Here either $\co_{1}^{\deriv^{\!\gamma}}\!(V) = 0$, or by Proposition \ref{careful} we have $\co_{1}^{\deriv^{\!\gamma}}\!(V) \neq 0$ and
\begin{align*}
  \deg \deriv\!\left(
 \co_{1}^{\deriv^{\!\gamma}}\!(V)
 \right) + \gamma 
 &\leq \deg \co_{1}^{\deriv^{\!\gamma}}(V) + \gamma - 1
  \\
  &\leq \max\!\left\{h^{j}(V) + j + 1 : h^{j}(V) \geq 0 \text{ and }
  0 \leq j \leq \gamma-1\right\} - 1
  \\
  &= \max\!\left\{h^{j}(V) + j : h^{j}(V) \geq 0 \text{ and }
  0 \leq j \leq \gamma-1\right\}
  \\
  &< \reg(V) = \rho
\end{align*}
by the definition of $\gamma = \crit(V)$. Thus we always have $\deg \deriv\!\left(
 \co_{1}^{\deriv^{\!\gamma}}\!(V)
 \right) < \rho - \gamma$ and $(\dagger\dagger)$ in degree $\rho-\gamma$ becomes an isomorphism
\begin{align*}
  \co_{1}^{\deriv^{\!\gamma+1}}\!(V)_{\rho-\gamma} \cong
  \kert\!\left(
 \deriv^{\!\gamma}(V) 
 \right)_{\rho - \gamma} 
\end{align*}
of $\sym{\rho-\lambda}$-modules, bridging the module in (2) with the one in (3) when $i=1$. We shall complete the proof by showing
\begin{align*}
  \co_{i}^{\deriv^{\!\gamma+i}}\!(V)_{\rho-\gamma} \cong
  \co_{i-1}^{\deriv^{\!\gamma+i-1}}\!(V)_{\rho-\gamma}
\end{align*}
for every $i \geq 2$. To that end either $\co_{i}^{\deriv^{\!\gamma+i-1}}\!(V) = 0$, or by part (2) of Proposition \ref{careful} we have $\co_{i}^{\deriv^{\!\gamma+i-1}}\!(V) \neq 0$ and
\begin{align*}
 \deg \deriv\!\left(
 \!\co_{i}^{\deriv^{\!\gamma+i-1}}\!(V)
 \right) + \gamma + i
 &\leq \deg \co_{i}^{\deriv^{\!\gamma+i-1}}(V) + \gamma + i - 1
  \\
  &\leq \max\!\left\{h^{j}(V) + j + i : h^{j}(V) \geq 0 \text{ and }
  0 \leq j \leq \gamma-1\right\}
  \\
  &= \max\!\left\{h^{j}(V) + j : h^{j}(V) \geq 0 \text{ and }
  0 \leq j \leq \gamma-1\right\} + i
  \\
  &< \reg(V) + i = \rho + i
\end{align*}
by the definition of critical index. In any case we have 
\begin{align*}
 \deg \deriv\!\left(
 \!\co_{i}^{\deriv^{\!\gamma+i-1}}\!(V)
 \right) < \rho - \gamma
\end{align*}
and hence evaluating $(\dagger)$ in degree $\rho-\gamma$ results in an isomorphism
\begin{align*}
 \co_{i}^{\deriv^{\!\gamma+i}}\!(V)_{\rho-\gamma} 
 \cong \kert\!\left(
 \co_{i-1}^{\deriv^{\!\gamma+i-1}}(V) 
 \right)_{\rho-\gamma}
\end{align*}
of $\sym{\rho-\gamma}$-modules. Here by part (2) of Proposition \ref{careful}, the $\FI$-module $W := \co_{i-1}^{\deriv^{\!\gamma+i-1}}(V)
$ satisfies
\begin{align*}
 \deg W + \gamma + i - 1
 &\leq \max\!\left\{
 h^{j}(V) + j + i - 1 : h^{j}(V) \geq 0 \text{ and }
 0 \leq j \leq \gamma
 \right\} = \rho + i - 1
\end{align*}
by the definition of $\gamma = \crit(V)$. Thus \(\deg W \leq \rho-\gamma\) and hence $(\kert W)_{\rho-\gamma} = W_{\rho-\gamma}$.
\end{proof}

\begin{proof}[Proof of \emph{\textbf{Theorem \ref{width}}}]
 For every $i \geq 1$, whenever $\co_{i}^{\deriv^{\!a}}(V) \neq 0$ we necessarily have $a \geq 1$ and part (2) of Proposition \ref{careful} yields
\begin{align*}
  \deg \co_{i}^{\deriv^{\!a}}(V) + a 
  &\leq \max\!\left\{h^{j}(V) + j + i : h^{j}(V) \geq 0 \text{ and }
  0 \leq j \leq a-i\right\}
  \\
  &= \max\!\left\{h^{j}(V) + j : h^{j}(V) \geq 0 \text{ and }
  0 \leq j \leq a-i\right\} + i
  \\
  &\leq \reg(V) + i
\end{align*} 
where the last line is by Theorem \ref{nss}. As a result, $\reg(V) + i$ is an upper bound for the set
\begin{align*}
 \mathbf{D}_{i}(V) := \left\{
  \deg\!\left(
           \co_{i}^{\deriv^{\!a}}(V)
        \right) 
  + a \,:\, \co_{i}^{\deriv^{\!a}}(V) \neq 0
  \right\},\,
\end{align*}
so that 
\begin{align*}
 \width{i}{V} = \sup \mathbf{D}_{i}(V) \leq \reg(V) + i \, .
\end{align*}
On the other hand, setting $\gamma := \crit(V)$ and invoking Proposition \ref{all-iso}, for every $i \geq 1$ we have $\co_{i}^{\deriv^{\!\gamma+i}}(V) \neq 0$ and
\begin{align*}
 \deg\!\left(\co_{i}^{\deriv^{\!\gamma+i}}(V)\right) &\geq \reg(V) - \gamma
 \, ,
 \\
 \deg\!\left(\co_{i}^{\deriv^{\!\gamma+i}}(V)\right) + \gamma + i &\geq \reg(V) + i \, .
\end{align*}
But the left hand side of the last inequality is an element of $\mathbf{D}_{i}(V)$, whence 
\begin{align*}
 \width{i}{V} = \sup \mathbf{D}_{i}(V) \geq \reg(V) + i \, .
\end{align*}
\end{proof} 

\bibliographystyle{hamsalpha}
\bibliography{stable-boy}

\providecommand{\bysame}{\leavevmode\hbox to3em{\hrulefill}\thinspace}
\providecommand{\MR}{\relax\ifhmode\unskip\space\fi MR }
\providecommand{\MRhref}[2]{%
  \href{http://www.ams.org/mathscinet-getitem?mr=#1}{#2}
}
\providecommand{\href}[2]{#2}
\begin{thebibliography}{Ram18}

\bibitem[Bah23]{bahran-polynomial}
Cihan Bahran, \emph{{Polynomial conditions and homology of
  {\textbf{\emph{FI}}}-modules}}, Pacific Journal of Mathematics \textbf{324}
  (2023), no.~2, 207--226. \MR{4619850}

\bibitem[Bah24]{bahran-reg}
\bysame, \emph{{Regularity and stable ranges of
  $\text{\textbf{\emph{FI}}}$-modules}}, Journal of Algebra \textbf{641}
  (2024), 429--497. \MR{4673996}

\bibitem[CE17]{ce-homology}
Thomas Church and Jordan~S. Ellenberg, \emph{{Homology of
  {\textbf{\emph{FI}}}-modules}}, Geometry \& Topology \textbf{21} (2017),
  no.~4, 2373--2418. \MR{3654111}

\bibitem[NSS18]{nss-regularity}
Rohit {Nagpal}, Steven~V. {Sam}, and Andrew {Snowden}, \emph{{Regularity of
  \(\mathbf {FI}\)-modules and local cohomology}}, {Proceedings of the American
  Mathematical Society} \textbf{146} (2018), no.~10, 4117--4126.

\bibitem[Ram18]{ramos-fig}
Eric Ramos, \emph{{Homological invariants of {\textbf{\emph{FI}}}-modules and
  {\textbf{\emph{FI}$_{G}$}}-modules}}, {Journal of Algebra} \textbf{502}
  (2018), 163--195.

\bibitem[RW17]{rw-wahl-stab}
Oscar {Randal-Williams} and Nathalie {Wahl}, \emph{{Homological stability for
  automorphism groups}}, {Advances in Mathematics} \textbf{318} (2017),
  534--626.

\end{thebibliography}

\end{document}